\documentclass[12pt]{amsart}
\usepackage{amsmath}
\usepackage{amsfonts, amssymb, mathrsfs,bm}
\usepackage{amsthm, upref, paralist,setspace}
\usepackage{graphicx}
\usepackage{hyperref}
\usepackage[usenames, dvipsnames]{color}
\usepackage[margin=1in]{geometry}
\usepackage{aliascnt}
\usepackage{verbatim}

\allowdisplaybreaks[4]

\newcommand{\BR}{\mathbb{R}}

\newcommand{\SL}{\sum\limits}

\newcommand{\al}{\alpha}
\newcommand{\be}{\beta}
\newcommand{\ga}{\gamma}
\newcommand{\de}{\delta}

\newcommand{\ME}{\mathbb E}

\newcommand{\CF}{\mathcal F}
\newcommand{\CH}{\mathcal H}

\newcommand{\MP}{\mathbb P}

\newcommand{\MQ}{\mathbb Q}

\newcommand{\CW}{\mathcal W}

\newcommand{\Oa}{\Omega}

\newcommand{\si}{\sigma}

\newcommand{\pa}{\partial}
\renewcommand{\phi}{\varphi}

\newcommand{\norm}[1]{\lVert#1\rVert}
\renewcommand{\comment}[1]{}

\newcommand{\md}{\mathrm{d}}

\renewcommand{\d}{{\rm d}}

\DeclareMathOperator{\mes}{\text{Leb}}

\DeclareMathOperator{\dist}{dist}

\theoremstyle{plain}
\newtheorem{thm}{Theorem}[section]

\newtheorem{lemma}[thm]{Lemma}

\numberwithin{equation}{section}

\theoremstyle{definition}
\newtheorem{defn}{Definition}[section]
\newtheorem{asmp}{Assumption}[section]

\theoremstyle{remark}
\newtheorem{rmk}{Remark}[section]

\begin{document}

\title[Talagrand Concentration Inequalities for SPDE]{Talagrand Concentration Inequalities for\\ Stochastic Partial Differential Equations}

\author{Davar Khoshnevisan and Andrey Sarantsev}

\address{ Department of Mathematics, University of Utah, Salt Lake City}

\email{davar@math.utah.edu}

\address{ Department of Mathematics and Statistics, University of Nevada, Reno} 
\email{asarantsev@unr.edu}

\date{Last version on November 24, 2018, based on the earlier version dated April 25, 2018.}

\begin{abstract}
One way to define the concentration of measure phenomenon is via Talagrand inequalities, also called transportation-information inequalities. That is, a comparison of the Wasserstein distance from the given measure to any other absolutely continuous measure with finite relative entropy. Such transportation-information inequalities were recently established for some stochastic differential equations. Here, 	we develop a similar theory for some stochastic partial differential equations. 
\end{abstract}



\maketitle

\thispagestyle{empty}

\section{Introduction}  
Let $(E\,, \rho)$ be a metric space with a Borel $\sigma$-algebra $\mathfrak B(E)$. 
Consider a Borel probability measure $\MQ$ on $E$. Define 
$A_r := \{x \in E:\, \dist(x\,, A) \le r\}$ for every Borel set 
$A \subseteq E$ and all $r > 0$; also, let $A_r^c := E\setminus A_r$ denote 
the complement of $A_r$ in $E$. Now, consider 
\begin{equation}\label{eq:alpha}
	\al(r) := \sup\left\{\MQ(A^c_r):\, A \in \mathfrak B(E),\, \MQ(A) \ge \tfrac12\right\}.
\end{equation}
The {\it concentration of measure phenomenon} is the property that
$\al(r)\approx0$ when $r\gg1$. The quality of the concentration of $\MQ$  depends
on the rate at which $\alpha(r)$ tends to zero as $r\to\infty$.

\smallskip

L\'evy initiated the study of concentration of measure by verifying that the normalized Lebesgue measure on $\mathbb{S}^{n-1}$ concentrates \cite[\S1.1]{LedouxBook}. The theory reached new heights in the work of Milman on the local theory of Banach spaces. Later on, Talagrand investigated the concentration of  product measures \cite{Talagrand2, Talagrand3, Talagrand5, Talagrand6}.
These references include also detailed pointers to the earlier parts of the literature. One of Talagrand's novel ideas in this direction was that ``a Lipschitz-continuous function of many variables, which does not depend much on any single variable, is nearly a constant''; see \cite{Bobkov2, NewBook,LedouxBook}. 

\smallskip

Concentration of measure is  related closely to the log-Sobolev and Poincar\'e inequalities 
\cite{Bobkov1, C1, C2, C3, Gozlan, Otto, VillaniBook}, with intimate connections
to information theory \cite{Bobkov3}, optimal transport  \cite{VillaniBook}, 
random matrices \cite{Bobkov4},  random graphs \cite[Chapter 2]{Graph}, 
and large deviations  \cite{DemboBook, Dembo}. Concentration of measure 
has been successfully applied to problems in stochastic finance \cite{Lacker},  
model selection in statistics \cite{Model}, and to the analysis of randomized 
algorithms \cite{Algorithm}. Among other things, concentration of measure has 
been established for the law of a large family of discrete-time Markov chains 
\cite{Marton1, Paulin, Samson},  discrete-time stationary processes \cite{Marton2},  
the solution of  a nice stochastic differential equation (SDE) \cite{LogConcave, Djellout, Pal, PalShkolnikov,Ustunel}, 
and for the law of the solution of a stochastic partial differential equation (SPDE) that is driven by a centered Gaussian 
noise that is white in time and whose spatial correlation operator is trace class 
\cite{Ustunel}. Since the latter SPDEs are  approximately finite-dimensional SDEs, 
it might be possible to derive concentration of measure for such SPDE from 
concentration for SDE. By contrast, our aim is to prove the concentration of 
measure for the law of the solution of a parabolic SPDE driven by space-time white noise.

\smallskip

To simplify our exposition, let us choose and fix two real numbers $D,T>0$,
a function $u_0\in L^\infty[0\,,D]$, and a second-order differential operator $\mathscr{L}$
that acts on $\varphi\in C^\infty([0, D])$ via
\[
    (\mathscr{L} \varphi)(x) := \tfrac12a^2(x)\varphi''(x) + b(x)\varphi'(x)\qquad
    \text{for all}\quad  x\in [0, D].
\]
When there is also a temporal
variable $t$, the ``prime'' continues to represent differentiation with respect
to the spatial variable $x$. Here, $a, b$ satisfy the following assumption:

\begin{asmp} $a,b\in C^\infty([0\,, D])$, and $a$ is  bounded
uniformly away zero and infinity.
\label{asmp:ab}
\end{asmp}

Consider also two measurable
functions $\sigma,g:[0\,,T]\times[0\,,D]\times\BR\to\BR$
such that $\BR\ni u\mapsto \sigma(t\,,x\,,u)$ and $\BR\ni u\mapsto g(t\,,x\,,u)$ are
Lipschitz continuous uniformly for $(t\,,x)\in[0\,,T]\times[0\,,D]$, and $\sigma$
is bounded. Throughout, we work on a filtered probability space 
$(\Oa\,, \mathcal F\,, \{\mathcal F_t\}_{t \in [0, T]}, \mathbf P)$ where the filtration 
$\{\mathcal F_t\}_{t \in [0, T]}$ is assumed to be  right continuous, $\mathcal F_0$ 
is assumed to be augmented with all $\mathbf P$-null sets, and $\mathcal F_T := \mathcal F$. 

With the background notation under way,
let us consider the SPDE:
\begin{equation}\label{eq:SPDE-basic}
	\frac{\pa}{\partial t}u(t\,, x) = (\mathscr{L} u)(t\,, x) +  g(t\,, x\,, u(t\,, x)) + \si(t\,,x\,, u(t\,, x))\,
	\frac{\pa^2}{\pa t\,\pa x} W(t\,, x),
\end{equation}
for $(t\,,x)\in(0\,,T]\times(0\,,D)$, subject to initial data $u_0$ and one of the
following boundary conditions on $[0\,,D]$:

\smallskip

\begin{compactenum}[(a)]
    \item \textbf{(Homogeneous Dirichlet).} $u(t\,,0)= u(t\,,D)= 0$ for all $t\in(0\,,T)$;
    \item \textbf{(Homogeneous Neumann).} $u'(t\,,0)=u'(t\,,D)=0$ for all $t\in(0\,,T)$; or
    \item \textbf{(Periodic Boundary).} $u(t\,,0)=u(t\,,D)$ and $u'(t\,,0)=u'(t\,,D)$ 
    	for all $t\in(0\,,T)$.
\end{compactenum}

\smallskip

The forcing term $W:=\{W(t\,,x)\}_{t\in[0,T],x\in[0,D]}$ of \eqref{eq:SPDE-basic}
denotes the \emph{two-parameter Brownian sheet}; that is, 
$W$ is a mean-zero Gaussian process with
\[
	\text{\rm Cov}[W(t\,,x)\,,W(s\,,y)] =
	(s\wedge t)(x\wedge y)\qquad
	\text{for all}\quad s\,,t\in[0\,,T]\quad \mbox{and}\quad x\,,y\in[0\,,D].
\]
It follows easily from the above that the weak derivative
$\xi:=\pa^2 W/(\pa t\,\pa x)$ is {\it space-time white noise};
that is, $\xi$ a generalized, centered, Gaussian random field with covariance measure
\begin{equation}
\label{eq:white-noise}
	\text{\rm Cov}\left[ \xi(t\,,x)\,,
	\xi(s\,,y)\right] =\delta_0(t-s)\,\delta_0(x-y),
\end{equation}
where $\de_0$ stands for the Dirac delta function centered at zero. 
This space-time white noise is adapted to the filtration $\{\mathcal F_t\}_{t \in [0, T]}$ 
and is generated by it: That is, for every $t \in [0\,, T]$, the $\sigma$-algebra 
$\mathcal F_t$ is generated by $\{W(s\,, x)\}_{s \in [0, t],\, x \in [0, D]}$, 
followed by the standard procedures of making the filtration right-continuous and augmented by 
$\MP$-null sets. It is well known that \eqref{eq:SPDE-basic} has a unique
predictable solution $u$ --- in the sense of Walsh \cite{WalshBook} --- 
that has H\"older-continuous trajectories.
See Walsh \cite[Chapter 3]{WalshBook} for the analysis of \eqref{eq:SPDE-basic}
in a specific case; the present, more general case follows from the theory of Dalang 
\cite{Dalang99}.
In particular, we mention that the \emph{solution} to \eqref{eq:SPDE-basic}
is understood in the following, mild, sense:
\begin{equation}\label{mild}\begin{split}
	u(t\,,x) &= \int_0^D G(t\,,x\,,y) u_0(y)\,\d y
		+\int_{(0,T)\times(0,D)} G(t-s\,,x\,,y) g(s\,, y\,, u(s\,,y))\,\d y\,\d s\\
	&\hskip1.5in+ \int_{(0,t)\times(0,D)} G(t-s\,,x\,,y)
		\,\si(s\,, y\,, u(s\,,y))\, W(\d s\,\d y),
\end{split}\end{equation}
where $G$ is the heat kernel for the operator $\mathscr{L}$ with the same boundary conditions 
as in \eqref{eq:SPDE-basic}; see also Dalang \cite{Dalang99,DalangSurvey}. The theory of
Walsh \cite[Chapter 3]{WalshBook} can be extended in a well-known, standard, way to deduce that
$u$ also satisfies the following moment bound:
\[
	\sup\limits_{t\in[0,T]}\sup\limits_{x\in[0,D]}
	{\mathbb E}\left( |u(t\,,x)|^p\right)<\infty
	\qquad\text{for all $p\in[1\,,\infty)$}.
\]

From now on, let $\MP$ denote the law of the solution $\{u(t, x)\}_{t\in[0\,,T],\, x \in [0, D]}$. 
In a standard way, we may view $\MP$ as a Borel-regular probability 
measure on the space $C([0\,,T]\times[0\,,D])$ of real-valued continuous 
functions on $[0\,,T]\times[0\,,D]$. We may also view $\MP$ as a Borel-regular 
probability measure on $L^p([0\,,T]\times[0\,,D])$ for every $p\in[1\,,\infty)$.

\smallskip

The following summarizes some of the main results of this paper in somewhat informal language.
More formal statements will come in due time: 

\begin{enumerate}
\item The measure $\MP$ concentrates as a Borel measure on $E=L^2([0\,,T]\times[0\,,D])$. 	
\item If $\si$ is a constant, then $\MP$  concentrates as a Borel measure on $E=C([0\,,T]\times[0\,,D])$. 
\end{enumerate}

It is well known that one can study concentration of measure by establishing 
{\it Talagrand concentration inequalities} (see \S4), otherwise known as 
{\it transportation-cost information inequalities} (TCI inequalities).
These are inequalities that compare Wasserstein distance with relative entropy. 
We verify our concentration results by showing that, in fact, $\MP$ satisfies a TCI inequality. 
A key step of the proof is to appeal to a suitable version of the Girsanov theorem. 
This is consistent with the use of the Girsanov theorem in the previous literature on 
concentration of measure for SDEs and SPDEs with regular noise 
\cite{LogConcave,Djellout,Pal,Ustunel}, and is intimately related to the earlier fact that 
the Girsanov theorem generally yields transportation inequalities in
the abstract Wiener space via $L\log L$-type entropy bounds, first discovered 
by Feyel and \"Ust\"unel in \cite{FeyelUstunel2002,FeyelUstunel2004}; see also \cite{UZBook}. We prove necessary technical results from scratch in Lemma~\ref{lemma:entropy} (Girsanov representation of equivalent measure) and Lemma~\ref{lemma:mgle} (martingale representation).

\smallskip

Our methods can readily be extended to study TCI inequalities for 
\eqref{eq:SPDE-basic} in case where $\mathscr{L}$ has another form
than the one studied here. One needs only a reasonable set of heat-kernel
estimates. The particular form of $\mathscr{L}$ is not germane
to the present discussion. An example of the kind of operator that can be studied
by the same methods that we employ is the fractional Laplacian 
$\mathscr{L} = -(-\Delta)^\gamma$, where $\ga \in(1\,,2)$ 
\cite{DalangFrac, Algeria, Nane}. Heat kernel estimates for this operator can be found in \cite{Chen00, Chen03}. 

\smallskip

For $d \ge 2$, the SPDE~\eqref{eq:SPDE-basic} does not have mild solutions as in~\eqref{mild} as classical functions, see 
\cite[pp. 31-32, Exercise 6.10]{Book}. However, it has solutions in Sobolev spaces. This theory was developed by Krylov in \cite[pp.231-233]{Krylov}, and in subsequent papers \cite[Subsection 2.2]{SimpleSPDE}, \cite{ColoredSPDE, Kim}. An interesting topic for future research would be to prove concentration inequalities for them in Sobolev norms, including the case of a colored noise instead of the white noise.

\subsection{Organization of the paper} In Section 2 we recall Talagrand concentration inequalities 
and state our main results: (a) Theorem~\ref{thm:main-1} for constant $\sigma$, 
and concentration in the space of continuous functions; and (b) Theorem~\ref{thm:main-2} 
for the general case, and concentration in the space $L^2$. Section 3 is devoted to 
proofs of these results. The Appendix contains the proof of a martingale representation 
theorem for space-time white noise. This sort of representation theorem is undoubtedly well known.
We include the proof as it is short and self contained.

\section{Concentration Inequalities for SPDE: Main Results}

\subsection{Background on Talagrand concentration inequalities} 

Recall from the Introduction that $(E\,, \rho)$ is a metric space with Borel $\sigma$-algebra 
$\mathfrak B(E)$. Fix a real number $p\ge1$, and recall that the {\it Wasserstein 
distance of order $p$} between two Borel probability measures $\MQ_1, \MQ_2$ on $E$ is defined as 
$$
	\CW_p(\MQ_1, \MQ_2) := \inf\limits_{\pi}\left[\int\left\{\rho(x\,, y)\right\}^p
	\pi(\d x\,\d y)\right]^{1/p},
$$
where the $\inf$ is taken over all {\it couplings} $\pi$
of $\MQ_1$ and $\MQ_2$. (A \emph{coupling} on $E$ is a Borel
probability measure $\pi$ on $E\times E$ whose 
marginal distributions are respectively $\MQ_1$ and $\MQ_2$.) 

The {\it relative entropy} 
$\CH(\MQ_2\mid\MQ_1)$ of $\MQ_2$ with respect to $\MQ_1$ is defined as follows: 
$$
	\CH(\MQ_2\mid\MQ_1) := \ME^{\MQ_2}\left[\log\frac{\md\MQ_2}{\md\MQ_1} \right]
	= \ME^{\MQ_1}\left[\frac{\md\MQ_2}{\md\MQ_1}\log\frac{\md\MQ_2}{\md\MQ_1}\right]
	\qquad\text{if}\quad \MQ_2 \ll \MQ_1,
$$
and $\CH(\MQ_2\mid\MQ_1) = \infty$ if $\MQ_2\not\ll\MQ_1$. See also
\cite{FeyelUstunel2002,FeyelUstunel2004}.
Here, we denote by $\ME^{\MQ}$ the expectation with respect to measure $\MQ$ for every
probability measure $\MQ$; that is, $\ME^{\MQ}f:=\int f\,\d\MQ$
for every bounded and measurable function $f:E\to\BR$.

\begin{defn} 
	We say that a Borel probability measure $\MQ_1$ satisfies the 
	{\it transportation-cost information (TCI) inequality of order} 
	$p$ with constant $C > 0$ when 
	\begin{equation}\label{eq:TCI}
		\CW_p(\MQ_1\,, \MQ_2) \le \sqrt{2C\CH(\MQ_2\mid\MQ_1)}.
	\end{equation}
	for every Borel probability measure 
	$\MQ_2$ on $E$. Throughout, we let $T_p(C)$ denote the set of all Borel probability measures 
	$\MQ_1$ that satisfy \eqref{eq:TCI} for every Borel probability measure $\MQ_2$ on $E$.
\end{defn}

According to H\"older's inequality,
$$
	T_{p'}(C) \subseteq T_p(C)\qquad\text{whenever $1\le p\le p'$ and $C>0$.}
$$
The following result from \cite{Marton} (see also \cite[p.\ 118]{LedouxBook}) 
relates TCI inequalities to the concentration of measure phenomenon: 
If $\MQ \in T_1(C)$, then the function $\al$ defined in~\eqref{eq:alpha} satisfies 
$$
	\al(r) \le {\rm e}^{- r^2/(8C)}
	\hskip1in\mbox{for}\ \ r \ge r_0 := 2\sqrt{2C\ln 2}.
$$
It is known that $\MQ\in T_1(C)$ for some constant $C > 0$ iff 
$\MQ$ has a \emph{sub-Gaussian tail}; that is,
$$
	\int_E{\rm e}^{[\rho(x_0, x)]^2/(2C)}\,\md\MQ(x) < \infty,
$$
for some, hence   all, $x_0 \in E$.
See  \cite{Bobkov1, Djellout}. Equivalently, $\MQ\in T_1(C)$ iff
\begin{equation}
	\label{eq:BG99}
	\int {\rm e}^{af}\mathrm{d}\MQ\le {\rm e}^{a^2C/2},
\end{equation}
for all $a\in\BR$ and every $1$-Lipschitz function $f:E\to\BR$ such that $\ME^{\MQ}f = 0$; here, $1$-Lipschitz means that satisfies
$|f(x)-f(y)|\le\rho(x\,,y)$ for all $x,y\in E$. 

\smallskip

It follows immediately from \eqref{eq:BG99} that every probability measure 
$\MQ\in \cup_{C>0}T_1(C)$ has sub-Gaussian tails.  
In particular, compactly-supported Borel probability measures are in $\cap_{C>0}T_1(C)$.
By contrast, similar descriptions of $T_p(C)$ for $p>1$ require more subtle analysis. 
For example, when $p>1$, the space $T_p(C)$ does not even
contain a non-trivial Bernoulli measure.
The space $T_2(C)$ has the particularly important property of {\it tensorization:}
If $\MQ_1$ and $\MQ_2$ are in $T_2(C)$, then $\MQ_1 \times \MQ_2$ is in $T_2(C)$
(as a Borel probability measure on $E\times E$, of course). This property
sets $T_2(C)$ apart as an important family of probability measures, and hence plays a central role
in the sequel.

\subsection{Main results}
Let us state the main results of this article. 
Recall from the Introduction the following assumptions on the functions $g$ and $\sigma$:

\begin{asmp}
There exists a real number $L_g > 0$ such that for all $(t\,,x)\in[0\,,T]\times[0\,,D]$ and $u,v\in\BR$,
\begin{equation}
\label{eq:g-Lip}
|g(t\,, x\,, u) - g(t\,, x\,, v)| \le L_g|u-v|.
\end{equation}
\label{asmp:g}
\end{asmp}

\begin{asmp} There exist real numbers $ L_{\sigma}, K_{\sigma}>0$ such that for all 
$(t\,,x)\in[0\,,T]\times[0\,,D]$ and $u,v\in\BR$,
\begin{equation}\label{eq:sigma-Lip}
|\sigma(t\,, x\,, u) - \sigma(t\,, x\,, v)| \le L_{\sigma}|u-v|;\qquad |\sigma(t\,, x\,, u)| \le K_{\si}.
\end{equation}
\label{asmp:sigma}
\end{asmp}

Let $(t\,,x\,,y)\mapsto G_{\infty}(t\,, x\,, y)$ denote the usual [Gaussian] heat kernel of the operator 
$\mathscr{L}$ on the whole real line instead of $[0\,, D]$, with coefficients $a$ and $b$ continued 
to $\BR$ as follows: $a(x) = a(0)$ for $x \le 0$, $a(x) = a(D)$ for $x \ge D$; similarly for $b$. 
It is well known that:

\smallskip

\begin{compactenum}
	\item $G\le G_{\infty}$, in the case of Dirichlet boundary conditions;
	\item $G = G_{\infty} + G_0$ for a smooth and bounded function $G_0$, in the case of 
		Neumann or periodic boundary conditions. 
\end{compactenum}

\smallskip

Because the second-order differential operator $\mathscr{L}$ is in divergence form,
Aaronson-type heat-kernel estimates imply that
$\mathcal{G}_{T,\al}<\infty$ for $\al\in[1\,,2)$, where
\begin{equation}\label{eq:G-const}
	\mathcal{G}_{T,\al} := \int_0^T [H(t)]^\alpha\,\md t < \infty,\quad\mbox{and}\quad
	H(t) := \sup\limits_{x \in [0\,, D]}\int_0^D [G(t\,, x\,, y)]^2\,\md y.
\end{equation}
See Bass \cite[Chapter 7, Theorem 4.3]{BassBook}, or \cite{HeatKernel}.
Consequently,
\begin{equation}\label{eq:GT2}
	\mathcal{G}_T := \sup\limits_{x \in [0, D]}\int_{(0,T)\times(0,D)}[G(t\,, x\,, y)]^2\,\md y\,\md t
	<\infty.
\end{equation}
As mentioned earlier, one can consider concentration inequalities in 
different Banach  spaces. Since $u$ is continuous, we can for example consider concentration
in the space of continuous functions on $(0\,,T)\times(0\,,D)$, endowed with norm,
\begin{equation}\label{eq:norm}
    \norm{u}_{\infty, T} := \max\limits_{(t, x) \in (0,T)\times(0,D)}|u(t\,, x)|.
\end{equation} 

\begin{thm}\label{thm:main-1}
For $\si \equiv 1$, under Assumptions~\ref{asmp:ab} and~\ref{asmp:g}, the law  $\MP$ of the solution $u$ of
    \eqref{eq:SPDE-basic}, viewed as a Borel probability measure on 
    $E=C([0\,, T]\times [0, D])$,  is in $T_2(C_{\infty})$
    with respect to the norm~\eqref{eq:norm}, with the constant
    \begin{equation}\label{eq:constant-1}
        C_{\infty} := 2\mathcal{G}_T{\rm e}^{2L_g^2T^2}.
    \end{equation}
\end{thm}

\begin{rmk}
    Choose and fix an arbitrary $\eta>0$. A simple adaptation of the proof of 
    Theorem \ref{thm:main-1} shows that we may replace
    the condition  $\si\equiv1$ with $\si\equiv \eta$.
\end{rmk}

In the non-constant
case, we instead view the random continuous function $u$ as a random element in the space 
$L^2([0\,, T]\times[0\,, D])$, endowed with the norm $\norm{\cdot}_{T, 2}$, where
\begin{equation}\label{eq:norm-2}
	\norm{u}_{T, 2}^2 := \int_{(0,T)\times(0,D)} [u(t\,, x)]^2\,\md x\,\md t.
\end{equation}

\begin{thm}\label{thm:main-2}
	Under Assumptions~\ref{asmp:ab},~\ref{asmp:g},~\ref{asmp:sigma}, for every $\al\in(1\,,2)$,
	the probability measure $\MP$, viewed as a Borel probability measure in the space 
	$L^2([0\,, T]\times [0\,, D])$,  is in 
	$T_2(C_{2,\al})$ with respect to the norm~\eqref{eq:norm-2}, where we define $\be$ from  $\alpha^{-1}+\beta^{-1}=1$, and 	\begin{equation}\label{eq:constant-2}
		C_{2, \al} := TD3^{2-\be^{-1}}K_{\sigma}^{2}\mathcal{G}_T\exp\left[T\be^{-1}
		3^{2\be - 1}L_{\sigma}^{2\be}\left(\mathcal{G}_{T,\al}^{\be/\al} +  
		\mathcal{G}_T^{\be}T^{\be/\al}\right)\right].
	\end{equation}
\end{thm}

\section{Proofs}

\subsection{Representation of an equivalent measure}
The following lemma essentially describes all probability measures 
$\MQ \ll \MP$ on $C([0\,, T]\times [0\,, D])$ or $L^2([0\,, T]\times[0\,, D])$. 
This lemma is an analogue of the result \cite[Theorem 5.6, (5.7)]{Djellout}, 
though it is applicable to the setting of space-time white noise instead of that of finite-dimensional Brownian motion. 

Take any $\MQ \ll \MP$ on $L^2([0\,, T]\times[0\,, D])$. 
The Radon--Nikod\'ym derivative 
$\md\MQ/\md\MP$ is a function $L^2([0\,, T]\times [0\,,D]) \to \mathbb R$. 
Therefore, we can realize the random variable $\xi := (\md\MQ/\md\MP)(u)$ 
on the filtered probability space 
$(\Oa\,, \CF\,, \{\CF\}_{0 \le t \le T}, \mathbf{P})$. 
Define a new probability measure $\mathbf{Q}$ on this probability space by 
\[
	\md\mathbf{Q} = \xi\md\mathbf{P},
\]
and let 
$\tilde{\mathbb E} := \mathbb E^{\mathbf{Q}}$ denote the expectation
with respect to this new measure. Consider the  nonnegative 
$\MP$-martingale $M$ that is defined by
\begin{equation}
	\label{eq:Radon-Nikodym}
		M(t) := \ME^{\MP}(\xi\mid\CF_t)
		 = \left.\frac{\md\mathbf{Q}}{\md\mathbf{P}}\right|_{\CF_t}\ \quad\text{for all $t \in [0\,, T]$.}
\end{equation}
General theory assures us that the process $M$ is a.s.\ continuous (up to a modification, which we adopt)
with respect to $\mathbf{P}$, and therefore also
with respect to $\mathbf{Q}$. 

\begin{lemma}\label{lemma:entropy}
	There exists an adapted (jointly measurable) process 
	$X = \{X(s\,, x)\}_{(s, x) \in [0, T] \times [0, D]}$ such that, $\mathbf Q$-a.s. for all $t \in [0\,, T)$, 
	\begin{equation}\label{eq:finite-moment-prove}
		\norm{X}^2_{t, 2} := \int_{(0,t)\times(0,D)}X^2(s, x)\,\md x\,\md s < \infty
	\end{equation}
	and $\widetilde{W} : [0\,, T]\times[0\,, D] \to \mathbb R$, defined by 
	\begin{equation}\label{eq:new-white-noise}
		\widetilde{W}(t\,, x) := W(t\,, x) - \int_{(0,t)\times(0,x)}X(s\,, y)\,\md y\,\md s,
	\end{equation}
	is a Brownian sheet under the measure $\mathbf{Q}$. Moreover,
	\begin{equation}\label{eq:measure-derivative}
		M(t) = \exp\left(\int_{[0, t]\times[0, D]}
		X(s\,, x)\,W(\md s\, \md x) - \frac12\norm{X}^2_{t, 2}\right)
		\qquad\text{$\mathbf{Q}$-a.s.},
	\end{equation}
	and
	\begin{equation}\label{eq:entropy-expression}
		\CH(\MQ\mid\MP) = \tfrac12\tilde{\ME}\left(\norm{X}^2_{T, 2}\right).
	\end{equation}
\end{lemma}

\begin{proof} 	Let $\tau := \inf\{t \ge 0: M(t) = 0\}\wedge T$, with the convention 
	$\inf\varnothing := \infty$. In light of~\eqref{eq:Radon-Nikodym}, $\mathbf{Q}\{\tau = T\}=1$. 
	Up until the stopping time $\tau$, 
	the martingale $M$ can be represented as the stochastic exponential of another 
	continuous local martingale $N$: 
	\begin{equation}\label{eq:M-exponential}
		M(t) = {\rm e}^{N(t) - \frac12\langle N\rangle_t}
		\qquad\text{for all $t\in[0\,,\tau)$}.
	\end{equation}
	Let $\tau_n := \inf\{t \ge 0:\ N(t) \notin (n^{-1}, n)\}$. Then $\tau_n \uparrow \tau$ a.s. Observe that the stopped process 
	\begin{equation}\label{eq:stopped-N}
		N_n(t) := N(t\wedge\tau_n)\qquad(t \in [0\,, T])
	\end{equation}
	defines a square-integrable $\mathbf{P}$-martingale with respect to the 
	filtration $\{\CF_t\}_{t\ge0}$. By Lemma \ref{lemma:mgle} below, there exists a process
	$X_n \in L^2(\Oa\times[0\,, T]\times [0\,, D])$ such that $\mathbf{P}$-a.s.,
	\begin{equation}\label{eq:N-n}
		\int_{[0, t]\times[0, D]}X_n(s\,, x)\,W(\md s\,\md x) = N_n(t)
		\qquad\text{for all $t \in [0\,, T],$}
	\end{equation}
	for every positive integer $n$.
	Without loss of generality, we can define $X_n(t) \equiv 0$ for $t > \tau_n$. 
	Now,  the optional stopping theorem ensures that, for every pair of integers $n > m$, 
	$$
		\ME\left(N_n(t)\mid\CF_{\tau_m}\right) = 
		N_n\left(t\wedge\tau_m\right) = N(t\wedge\tau_n\wedge\tau_m) =
		N(t\wedge\tau_m) = N_m(t)\qquad \mathbf{P}\mbox{-a.s.}
	$$
		Since $X_n(t)\bm{1}_{\{t \le \tau_m\}}$ is $\CF_{\tau_m}$-measurable, by uniqueness of the representation from Lemma~\ref{lemma:mgle}
	\begin{equation}\label{eq:consistency}
		X_n\left(t\wedge\tau_m\right) = X_m\left(t\wedge\tau_m\right)
		\quad\text{when}\quad n > m,\quad \mathbf{P}\, \mbox{-a.s.}
	\end{equation}
	As shown above, 
	\begin{equation}\label{eq:conv-of-stopping-times}
		\mathbf{Q}\left\{ \tau_n \uparrow \tau = T
		\text{ as $n\uparrow\infty$}\right\}=1,
	\end{equation}
	and let
	\begin{equation}\label{eq:new-X}
		X(t) := X_n(t)
		\qquad\text{for all $t \le \tau_n$ and for all $n\ge1$.}
	\end{equation}
	The consistency relation \eqref{eq:consistency}
	ensures that the process $X$ from~\eqref{eq:new-X} is 
	defined coherently. The stochastic
	process $X$ is the process from the statement of the lemma.
	
	In accord with \eqref{eq:stopped-N}, \eqref{eq:consistency}, \eqref{eq:new-X}, 
	$N(t) = \int_{[0, t]\times[0, D]}X\,\md W$ for all $t \in [0\,, T].$
	Therefore, \eqref{eq:measure-derivative} follows from \eqref{eq:M-exponential}. 	For every $n \ge 1$, 
	\begin{equation}
		\tilde\ME \int_{(0,\tau_n)\times(0,D)}X^2(s\,, x)\,\md x\,\md s < \infty,
	\end{equation}
	Therefore,  
	\begin{equation}
	\label{eq:finite-moment}
		\mathbf{Q}\left\{ \int_{(0,\tau_n)\times(0,D)}X^2(s\,, x)\,\md x\,\md s < \infty
		\quad \mbox{for every $n\ge1$}\right\}=1.
	\end{equation}
	Choose and fix a time $t < T$. Because of~\eqref{eq:conv-of-stopping-times}, 
	$\mathbf{Q}$-a.s.\ there exists a random $n\ge1$ such that $t \le \tau_n$. Therefore,
	we can deduce \eqref{eq:finite-moment-prove} from \eqref{eq:finite-moment}. 
	
\smallskip

	Next we verify~\eqref{eq:new-white-noise}. Apply the Girsanov theorem 
	of da Prato and Zabczyk \cite[Theorem 10.1.4]{Prato} with 
	$\psi := X$, to see that $\widetilde{W}$ is indeed a space-time white noise and that 
	the Radon--Nikod\'ym formula \eqref{eq:measure-derivative} is valid. 
	This establishes \eqref{eq:new-white-noise}. 
	
	\smallskip
	
		Finally, let us show~\eqref{eq:entropy-expression}. From~\eqref{eq:Radon-Nikodym}, we can express
	\begin{equation}
	\label{eq:basic-entropy}
		\CH(\MQ\mid\MP) = \int_{L^2([0, T]\times[0, D])}
		\left[\frac{\md\MQ}{\md\MP}\ln\frac{\md\MQ}{\md\MP} \right] \md\MP 
		= \ME\left[\frac{\md\mathbf{Q}}{\md\mathbf{P}}\ln\frac{\md\mathbf{Q}}{\md\mathbf{P}}\right] 
		= \ME\left[M(T)\ln M(T)\right].
	\end{equation}
	Since $M$ is a nonnegative martingale, if $\tau < T$ then 
	$M(T) = M(\tau) = 0$, and  thus (with convention $0\ln(0) := 0$) we obtain the identity 
	$M(T)\ln M(T) = M(\tau)\ln M(\tau)$. Since $\tau_n \uparrow \tau$ and $M$
	is $\mathbf{P}$-a.s.\ continuous, it follows that
	\begin{equation}
	\label{eq:elementary}
		\lim_{n\to\infty}M(\tau_n)\ln M(\tau_n) = M(\tau)\ln M(\tau) = M(T)\ln M(T)
		\qquad\MP\text{-a.s.}
	\end{equation}
	Take expectation and interchange it with limits in~\eqref{eq:elementary}. 
	Indeed, the function $x \mapsto x\ln(x)$ is continuous and bounded from 
	below on $[0\,, \infty)$; it also is decreasing on $[0\,, {\rm e}^{-1}]$ and 
	increasing on $[{\rm e}^{-1}\,, \infty)$. Recall the definition of $\tau_n$ and 
	notice that the convergence  in~\eqref{eq:elementary}  is a.s.\ nondecreasing
	starting from $n \ge 3$. Combine~\eqref{eq:basic-entropy} 
	and~\eqref{eq:elementary}, and swap $\ME^{\MP}$ and $\lim_{n \to \infty}$
	in order to find that
	$$
		\CH(\MQ\mid\MP) = 
		\lim_{n \to \infty}\ME\left[M(\tau_n)\ln M\left(\tau_n\right)\right].
	$$
	Next we calculate the preceding expectation for every $n$. 
	
	Since $M(\tau_n)$ is the Radon--Nikod\'ym derivative of $\MQ$ over $\MP$ 
	on the $\sigma$-algebra $\mathcal F_{\tau_n}$, 
	\begin{align*}
		\ME\left[M(\tau_n)\ln M\left(\tau_n\right)\right] &= 
			\tilde{\ME}\left[\ln M\left(\tau_n\right)\right]  
			= \tilde{\ME}\left[N\left(\tau_n\right) - 	\tfrac12\langle N\rangle_{\tau_n}\right] 
	 \\  & =  \tilde{\ME}\left[\int_{[0, \tau_n]\times[0, D]}X(s\,, x)
			\,W(\md s\, \md x) - \frac12
			\int_{[0, \tau_n]\times[0, D]}X^2(s\,, x)\,\md s\,\md x\right] \\ 
		& =  \tilde{\ME}\left[\int_{[0, \tau_n]\times[0, D]}
			X(s\,, x)\,\widetilde{W}(\md s\, \md x) + \frac12\int_{[0,{\tau_n}]\times(0,D)}
			X^2(s\,, x)\,\md x\,\md s\right] \\ 
		& = \frac12  \tilde{\ME}\left[\int_{[0,\tau_n]\times(0,D)}
			X^2(t\,, x)\,\md x\,\md s\right].
	\end{align*}
	We have used the fact that , since $\tilde{W}$ is a $\mathbf{Q}$-Brownian sheet,
	\eqref{eq:finite-moment} ensures that the stochastic integral with respect to the 
	corresponding white noise has mean zero. The limit as $n \to \infty$ is equal to $\frac12\tilde{\ME}(\norm{X}^2_{T, 2})$ by the monotone convergence theorem. This completes the proof of~\eqref{eq:entropy-expression}, and whence the Lemma~\ref{lemma:entropy}. 
\end{proof}

\subsection{Proof of Theorem~\ref{thm:main-1}}
For all $(t\,,x)\in[0\,,T]\times[0\,,D]$ let
\begin{equation}\label{eq:initial-integral}
	I(t\,, x) :=  \int_0^D G(t\,,x\,,y) u_0(y)\,\d y.
\end{equation}
Given an arbitrary Borel probability measure $\MQ \ll \MP$ on $C([0\,, T]\times [0\,, D])$, Lemma \ref{lemma:entropy} ensures that we can couple $(\MP\,, \MQ)$ as 
follows (using notation from Lemma~\ref{lemma:entropy}): On the filtered probability space $(\Oa, \CF, (\CF)_{0 \le t \le T}, \mathbf{Q})$, this is the law of a process $(u\,, v)$, where $u$ and $v$ solve the following equations:
\begin{align}\label{eq:u-coupling}\begin{split}
	&u(t\,, x) = I(t\,, x) + \int_{[0, t]\times[0, D]}G(t-s\,, x\,, y)
		\,\widetilde{W}(\md s\, \md y) \\ 
	& \hskip1in+ \int_{(0,t)\times(0,D)}
		G(t-s\,, x\,, y)g(y\,, u(s\,, y))\,\md y\,\md s \\ 
	& \hskip1in+ \int_{(0,t)\times(0,D)}
		G(t-s\,, x\,, y)X(s\,, y)\,\md y\,\md s,
\end{split}\end{align}
\begin{align}\label{eq:v-coupling}\begin{split}
	&v(t\,, x) =  I(t\,, x) + \int_{[0, t]\times[0, D]}G(t-s\,, x\,, y)
		\,\widetilde{W}(\md s\, \md y) \\ 
	& \hskip2in+ \int_{(0,t)\times(0,D)} G(t-s\,, x\,, y)g(y\,, v(s\,, y))\,\md y\,\md s.
\end{split}\end{align}
By the definition of the Wasserstein distance $\mathcal W_2$, 
\begin{equation}\label{eq:Wasserstein}
	\mathcal W_2(\MP\,, \MQ)  \le \left\{\tilde{\ME}\left[
	\max\limits_{t \in [0, T]}\max\limits_{x \in [0, D]}\left|
	u(t\,, x) - v(t\,, x) \right|^2
	\right] \right\}^{1/2}.
\end{equation}
In light of~\eqref{eq:entropy-expression} and~\eqref{eq:Wasserstein}, it remains to prove that
\begin{equation}\label{eq:final}
	\tilde{\ME}\left[ \max\limits_{t \in [0, T]}\max\limits_{x \in [0, D]}
	\left| u(t\,, x) - v(t\,, x)\right|^2\right] 
	\le C_{\infty}\tilde{\ME}\left(\norm{X}^2_{T, 2}\right).
\end{equation}
From~\eqref{eq:u-coupling} and~\eqref{eq:v-coupling}, we can represent $u(t\,, x) - v(t\,, x)$ as
\begin{align}\label{eq:difference}\begin{split}
	u(t\,, x) - v(t\,, x) &= \int_{(0,t)\times(0,D)}
		G(t-s\,, x\,, y)\left[g(y\,, u(s\,, y)) - g(y\,, v(s\,, y))\right]\,\md y\,\md s \\ 
	&\hskip1.5in + \int_{(0,t)\times(0,D)} G(t-s\,, x\,, y)X(s\,, y)\,\md y\,\md s.
\end{split}\end{align}
Since $(x_1 + x_2)^2 \le 2(x_1^2 + x_2^2)$ for all real numbers $x_1$ and $x_2$, 
\begin{align}\notag
	\left| u(t\,, x) - v(t\,, x)\right|^2 &\le 
		2\left[\int_{(0,t)\times(0,D)}
		G(t-s\,, x\,, y)\left[g(y\,, u(s\,, y)) - g(y\,, v(s\,, y))\right]\,\md y\,\md s\right]^2 \\ 
	& \hskip1in + 2\left[\int_{(0,t)\times(0,D)} G(t-s\,, x\,, y)X(s\,, y)\,\md y\,\md s\right]^2.
	\label{eq:657}
\end{align}
For every $t \in [0, T]$, define the quantity
\begin{equation}\label{eq:def-of-M}
	\nu(t) := \max\limits_{t \in [0, T]}\max\limits_{x \in [0, D]}\left|
	u(s\,, x) - v(s\,, x)\right|^2.
\end{equation}
To estimate the first term in the right-hand side of~\eqref{eq:657}, 
we apply the the Cauchy--Schwarz inequality with respect to the finite measure 
$G(t-s\,, x\,, y)\,\md y\,\md s$ on $[0\,, t]\times [0\,,D]$, whose total measure is not more than $t$,
in order to see that
\begin{equation}\label{eq:700}\begin{split}
	&\left[\int_{(0,t)\times(0,D)}
		G(t-s\,, x\,, y)\left[g(y\,, u(s\,, y)) - g(y\,, v(s\,, y))\right]\,\md y\,\md s\right]^2 \\ 
	&\hskip1in \le tL_g^2\int_{(0,t)\times(0,D)} G(t-s\,, x\,, y)\left[u(s\,, y) - v(s\,, y)\right]^2\,\md y\,\md s \\
	&\hskip1in \le TL_g^2\int_{(0,t)\times(0,D)} G(t-s\,, x\,, y)\nu(s)\,\md y\,\md s \\
	& \hskip1in \le TL_g^2\int_0^t \nu(s)\,\md s.
\end{split}\end{equation}
In the last line we used the fact that $\int_0^DG(r\,, x\,, y)\,\md y \le 1$ for all $r>0$
and $x\in(0\,,D)$. On one hand,
the preceding bounds the first term on the right-hand side of~\eqref{eq:657} from above.
On the other hand, the second term on the right-hand side of~\eqref{eq:657}
is not greater than $2\mathcal{G}_T\|X\|_{T,2}^2$ thanks to the Cauchy--Schwarz inequality 
and \eqref{eq:GT2}. Thus, we find that
\begin{equation}
\label{eq:almost-there}
	\left| u(t\,, x) - v(t\,, x) \right|^2 \le 2L_g^2T\int_0^t
	\nu(s)\,\md s + 2\mathcal{G}_T\norm{X}^2_{T, 2}. 
\end{equation}
Maximize over $(t\,,x)\in(0\,,T)\times(0\,,D)$, and 
then apply the expectation $\tilde{\ME}$ to see that
\begin{equation} \label{eq:Gronwall}
	\tilde{\ME}[\nu(t)] \le \, 2L_g^2T\int_0^t\tilde{\ME}[\nu(s)]\,\md s + 
	2\mathcal{G}_T\tilde{\ME}\left(\norm{X}^2_{T, 2}\right)
	\qquad\text{for all $t\in(0\,,T)$}.
\end{equation}
An appeal to the Gronwall inequality verifies \eqref{eq:final}, 
and hence also  Theorem~\ref{thm:main-1}. \qed

\subsection{Proof of Theorem~\ref{thm:main-2}} 
As we did in the proof of Theorem~\ref{thm:main-1}, by Lemma~\ref{lemma:entropy}, 
for every probability measure $\MQ \ll \MP$ on $L^2([0\,, T]\times[0\,, D])$, 
we can couple $(\MP\,, \MQ)$ as follows: Recall the definition of $I$ in~\eqref{eq:initial-integral}. Consider a stochastic process $(u\,, v)$ on the filtered probability space $(\Oa, \CF, (\CF)_{0 \le t \le T}, \mathbf{Q})$, defined as follows:
Under the measure , we have:
\begin{align}\label{eq:u-coupling-1}\begin{split}
	u(t\,, x)  = I(t\,, x) &+ \int_{(0,t)\times(0,D)}G(t-s\,, x\,, y)\si(y\,, u(s\,, y))\,
		\widetilde{W}(\md s\,\md y) \\ 
	& + 	\int_{(0,t)\times(0,D)}
		G(t-s\,, x\,, y)g(y\,, u(s\,, y))\,\md y\,\md s \\ 
	& + \int_{(0,t)\times(0,D)}G(t-s\,, x\,, y)\si(y\,, u(s\,, y))X(s\,, y)\,\md y\,\md s;
\end{split}\end{align}
\begin{align}\label{eq:v-coupling-1}\begin{split}
	v(t\,, x)  =  I(t\,, x) &+ \int_{(0,t)\times(0,D)}G(t-s\,, x\,, y)\si(y\,, v(s\,, y))
		\,\widetilde{W}(\md s\,\md y) \\ 
	& + \int_{(0,t)\times(0,D)}G(t-s\,, x\,, y)g(y\,, v(s\,, y))\,\md y\,\md s. 
\end{split}\end{align}
Then the law of $(u, v)$ in $L^2([0, T]\times[0, D])\times L^2([0, T]\times[0, D])$ is a coupling of $\MP$ and $\MQ$.  By definition of the Wasserstein distance $\mathcal W_2$, 
\begin{equation}\label{eq:Wasserstein-1}
	\mathcal W_2(\MP \,, \MQ) \le \left\{
	\tilde{\ME}\left[\int_{(0,T)\times(0,D)}
	\left|u(t\,, x) - v(t\,, x)\right|^2\,\md x\,\md t\right]\right\}^{1/2}.
\end{equation}
In light of~\eqref{eq:entropy-expression} and~\eqref{eq:Wasserstein}, 
Theorem~\ref{thm:main-1} will follow, once we prove that
\begin{equation}\label{eq:final-1}
	\tilde{\ME}\left[ \int_{(0,T)\times(0,D)}\left| u(t\,, x) - v(t\,, x)\right|^2\,\md x\,\md t 	\right] \le C_{2, \al}\tilde{\ME}\left(\norm{X}^2_{T, 2}\right).
\end{equation} 
We conclude by establishing \eqref{eq:final-1}. Thanks to~\eqref{eq:u-coupling-1} and~\eqref{eq:v-coupling-1}, 
\begin{align}\notag
	u(t\,, x) - v(t\,, x) =& \int_{(0,t)\times(0,D)}
		G(t-s\,, x\,, y)\left[g(y\,, u(s\,, y)) - g(y\,, v(s\,, y))\right]\,\md y\,\md s \\ 
	& + \int_{(0, t)\times(0, D)}G(t-s\,, x\,, y)\left[\si(y\,, u(s\,, y)) - \si(y\,, v(s\,, y))\right]\, 
		\widetilde{W}(\md s\,\md y) \label{eq:difference-1}\\ \notag
	& + \int_{(0,t)\times(0,D)} G(t-s\,, x\,, y)\,\si(y\,, u(s\,, y))\,X(s\,, y)\,\md y\,\md s. 
\end{align}
Apply to~\eqref{eq:difference-1} the elementary inequality 
$(x_1 + x_2 + x_3)^2 \le 3(x_1^2 + x_2^2 + x_3^2)$, valid for all real numbers $x_1, x_2, x_3$, in order to see that
\begin{align}\notag
 	| u(t\,, x)  - v(t\,, x) |^2 & \le 3\,\left[\int_{(0,t)\times(0,D)}
		G(t-s\,, x\,, y)\left[g(y\,, u(s\,, y)) - g(y\,, v(s\,, y))\right]
		\,\md y\,\md s\right]^2 \\ 
	& + 3[\eta(t\,, x)]^2 + 3\left[\int_{(0,t)\times(0,D)}
		G(t-s\,, x\,, y)\,\si(y\,, u(s\,, y))\,X(s\,, y)\,\md y\,\md s\right]^2,\label{eq:3-terms}
\end{align}
$$
\mbox{with}\qquad	\eta(t\,, x) := \int_{(0, t)\times(0, D)}
	G(t-s\,, x\,, y)\left[\si(y\,, u(s\,, y)) - \si(y\,, v(s\,, y))\right]\, 
	\widetilde{W}(\md s\, \md y).
$$
Apply the Cauchy--Schwartz inequality to the first term on the right-hand side of
\eqref{eq:3-terms} in order to deduce from \eqref{eq:g-Lip} that
\begin{align}
\label{eq:term1}
\begin{split}
	&\left[\int_{(0,t)\times(0,D)}
		G(t-s\,, x\,, y)\left[g(y\,, u(s\,, y)) - g(y\,, v(s\,, y))\right]
		\,\md y\,\md s\right]^2 \\ 
	& \le \int_{(0,t)\times(0,D)}
		[G(t-s\,, x\,, y)]^2\,\md y\,\md s \cdot 
		\int_{(0,t)\times(0,D)}
		\left[g(y\,, u(s\,, y)) - g(y\,, v(s\,, y))\right]^2\,\md y\,\md s  \\ 
	& \le \mathcal{G}_TL_g^2\int_{(0,t)\times(0,D)}\left[u(s, y) - v(s, y)\right]^2\,\md y\,\md s.
\end{split}
\end{align}
Similarly to~\eqref{eq:def-of-M}, for every $t \in [0\,, T]$, define
\begin{equation}
\label{eq:m}
	m(t) := \sup\limits_{s \in [0, t]}\sup\limits_{x \in [0, D]}\tilde{\mathbb E}\left(
	|u(s\,, x) - v(s\,, x)|^2\right).
\end{equation}
Owing to \eqref{eq:g-Lip} and the respective definitions of $H$  and $m$ from
\eqref{eq:G-const} and \eqref{eq:m}, the $\MQ$-expectation 
of the second term on the right-hand side of~\eqref{eq:3-terms} can be estimated as 
\begin{align}\notag
	\tilde{\ME}\left(|\eta(t\,, x)|^2\right) &= 
		\tilde{\ME}\int_{(0,t)\times(0,D)}
		[G(t-s\,, x\,, y)]^2
		\left[\sigma(s\,, y\,, u(s\,, y)) - \sigma(s\,, y\,, v(s\,, y))\right]^2
		\,\md y\,\md s \\ \notag
	&\le  L_{\sigma}^2\,\tilde{\ME}\int_{(0,t)\times(0,D)}
		[G(t-s\,, x\,, y)]^2 \left[u(s\,, y) - v(s\,, y)\right]^2\,\md y\,\md s \\ \notag
	& \le L_{\sigma}^2\int_{(0,t)\times(0,D)}
		[G(t-s\,, x\,, y)]^2 m(s)\,\md y\,\md s \\
	&= L_{\sigma}^2\int_0^tH(t-s)m(s)\,\md s = L_{\sigma}^2 (H*m)(t).
\label{eq:term2}
\end{align}
Finally, we estimate the third term on the right-hand side of \eqref{eq:3-terms}
by applying first the Cauchy--Schwarz inequality, and then \eqref{eq:g-Lip}, in order to find that
\begin{align}
\label{eq:term3}
\begin{split}
	&\left[\int_{(0,t)\times(0,D)}
		G(t-s\,, x\,, y)\si(y\,, u(s\,, y))X(s\,, y)\,\md y\,\md s\right]^2 \\ 
	&  \le K_{\si}^2\int_{(0,t)\times(0,D)}
		[G(t-s\,, x\,, y)]^2\,\md y\,\md s \int_{(0,t)\times(0,D)}
		[X(s\,, y)]^2\,\md y\,\md s  \\
	& \le K_{\si}^2\mathcal{G}_T\norm{X}^2_{T, 2}. 
\end{split}
\end{align}
Apply $\tilde{\ME}$ to both sides of \eqref{eq:3-terms}. Combine 
this with~\eqref{eq:term1},~\eqref{eq:term2},~\eqref{eq:term3} in order to see that
\begin{align}
\label{eq:Gronwall-1}
	m(t)  \le 3L_{\sigma}^2(H*m)(t) + 3\mathcal{G}_TL_g^2\int_0^tm(s)
	\,\md s  + 3K_{\sigma}^2\mathcal{G}_T\tilde{\mathbb E}\left(\norm{X}^2_{T, 2}\right).
\end{align}
Since $(x_1 + x_2 + x_3)^{\be} \le 3^{\be - 1}(x_1^{\be} + x_2^{\be} + x_3^{\be})$, the preceding
yields the following self-referential inequality for $m$:
\begin{align}
\label{eq:Gronwall-2}
	m^{\be}(t) 
	\le 3^{2\be-1}\left\{
	L_{\sigma}^{2\be} [(H*m)(t)]^{\be} + \mathcal{G}_T^{\be}L_g^{2\be}\left[\int_0^tm(s)\,\md s\right]^{\be}  + K_{\sigma}^{2\be}G^{\be}_T\left[\tilde{\ME}\norm{X}^2_{T, 2}\right]^{\be}\right\}.
\end{align}
Choose two positive H\"older-conjugates $\al^{-1} + \be^{-1} = 1$, and note that
\[
	[(H*m)(t)]^{\be}  \le \left[\int_0^tH^{\al}(s)\,\md s\right]^{\be/\al}
	\int_0^tm^{\be}(s)\,\md s \le \mathcal{G}_{T,\al}^{\be/\al}\int_0^tm^{\be}(s)\,\md s,
\]
and $[\int_0^tm(s)\,\md s]^{\be}  \le t^{\be/\al}\int_0^tm^{\be}(s)\,\md s 
\le T^{\be/\al}\int_0^tm^{\be}(s)\,\md s.$ 
Therefore, \eqref{eq:Gronwall-2} implies that
\begin{equation}
\label{eq:pre-Gronwall}
	m^{\be}(t) \le   3^{2\be - 1}L_{\sigma}^{2\be}\left(\mathcal{G}_{T,\al}^{\be/\al}
	+\mathcal{G}_T^{\be}T^{\be/\al}
	\right)\int_0^tm^{\be}(s)\,\md s  + 3^{2\be-1}K_{\sigma}^{2\be}
	\mathcal{G}_T^{\be}\left[\tilde{\ME}\left(\norm{X}^2_{T, 2}\right)\right]^{\be}.
\end{equation}
Thus,  Gronwall's inequality yields
\begin{equation}
\label{eq:final-ineq}
	m^{\be}(T) \le 3^{2\be-1}K_{\sigma}^{2\be}\mathcal{G}_T^{\be}
	\left[\tilde{\ME}\left(\norm{X}^2_{T, 2}\right)\right]^{\be}
	{\rm e}^{3^{2\be - 1}L_{\sigma}^{2\be}T(\mathcal{G}_{T,\al}^{\be/\al} +  
	\mathcal{G}_T^{\be}T^{\be/\al})}.
\end{equation}
Trivially estimating the integral in $\norm{u-v}^2_{T, 2}$, we get:
\begin{equation}
\label{eq:trivial}
\tilde{\ME}(\norm{u-v}^2_{T, 2}) \le TDm(T).
\end{equation}
Combining~\eqref{eq:final-ineq} and~\eqref{eq:trivial} and raising both sides to the power $1/\beta$, we finally obtain \eqref{eq:final-1}, and hence completes the proof of Theorem~\ref{thm:main-2}.

\section{Appendix: Martingale Representation} 
Fix a time horizon $T > 0$ throughout. 
The following theorem is an infinite-dimensional analogue of the 
classical martingale representation theorem 
\cite[Chapter 3, Theorem 4.15]{KSBook}. 
Related two-parameter martingale representation
theorems can  be found in \cite{CairoliWalsh,Ouvrard}, for example.
Though we are quick to point out that the following is a bona fide,
one-parameter martingale representation theorem for martingales
that are defined via the ``1-filtration'' of a space-time (2-parameter)
white noise.

\begin{lemma}
Every real-valued continuous square-integrable martingale 
$M = \{M(t)\}_{0 \le t \le T}$
can be represented as a stochastic integral
$M(t) = \int_{[0, t]\times[0, D]}X\,\md W$
for an adapted process $X$ on $[0\,, T]$ such that $X(t)\in L^2[0\,, D]$ 
for every $t\ge0$ and $\mathbb E(\norm{X}^2_{T, 2}) < \infty$. 
\label{lemma:mgle}
\end{lemma}

\begin{proof}
    Let $\{e_j\}_{j \ge 1}$ be an orthonormal basis of
    $L^2[0\,, D]$. For every integer $n\ge1$ define
    $\CF_n$ to be the $\si$-algebra generated by all random variables
    of the form $W(e_j\otimes \bm{1}_{[0, s]})$ as $j$ ranges in $\{1\,,\ldots,n\}$  and $s\in[0\,,T]$.
    For all $n\ge1$ and $t\in[0\,,T]$, define 
    \begin{equation}
    \label{eq:first-conditional}
    	M_n(t) := \ME(M(t)\mid \CF_n).
    \end{equation}
    Because $M$ is square-integrable, $\ME (|M(t)|^2) < \infty$ for every $t\in[0\,,T]$. 
    Thus $\{M_n(t)\}_{n \ge 1}$ is a martingale for every $t\in[0\,,T]$. 
    By L\'evy's martingale convergence theorem for discrete-time martingales, 
    \begin{equation}
    \label{eq:M-n-M}
    	\lim_{n\to\infty} M_n(t) = M(t)\ \mbox{a.s.\ and in}\ L^2\text{ for every fixed $t\in[0\,,T]$}.
    \end{equation}
     
    Define 
    $W_k(s) := W(e_k\otimes \bm{1}_{[0, s]})$ for every $s \in [0\,, t]$ and $k = 1, \ldots, n$, and note that
    $W_1,W_2,\ldots,W_n$ are i.i.d.\ Brownian motions. The random variable $M_n(t)$ is measurable with respect to $\CF_n(t) := \CF_n\cap\CF(t)$, and the latter defines a
    filtration generated by $n$ i.i.d. Brownian motions. From the right-continuity of $\{\CF(t)\}_{t \in [0, T]}$ follows the right-continuity of $\{\CF_n(t)\}_{t \in [0, T]}$. Therefore,  a finite-dimensional version of the martingale 
    representation theorem from \cite[Chapter 3, Theorem 4.15]{KSBook} implies that there exist $n$ processes 
    $X_{n, 1},\ldots,X_{n, n}$, all indexed by $[0\,, T]$
    that are adapted to the filtration $\{\CF_n(t)\}_{0 \le t \le T}$ and satisfy
    \begin{equation}\label{eq:mgle-repr}
    	M_n(t) = \SL_{k=1}^n\int_0^tX_{n, k}(s)\,\md W_k(s).
    \end{equation}
        For positive integers $n > m$ and for every $t \in [0\,, T]$, 
    \begin{equation}\label{eq:223}
   	 \ME(M_n(t)\mid \CF_m) = M_m(t).
    \end{equation}
    Consider the sum in \eqref{eq:mgle-repr} and write it as $\sum_{k=1}^m+\sum_{k=m+1}^n$.
    Because $\{W_k\}_{k > m}$ are independent of $\CF_m$ and $\{W_k\}_{1\le k \le m}$ 
    are $\CF_m$-measurable, it follows from \eqref{eq:223} that
    \begin{equation}
    \label{eq:new-mgle-repr}
    	\ME(M_n(t)\mid \CF_m) = \SL_{k=1}^m\int_0^t\ME(X_{n, k}(s)\mid\CF_m)\,\md W_k(s).
    \end{equation}
    We apply \eqref{eq:mgle-repr} once again, but this time replace $n$ by $m$ everywhere, in order
    to see that
    \begin{equation}
    \label{eq:mgle-repr-m}
    	M_m(t) = \SL_{k=1}^m\int_0^tX_{m, k}(s)\,\md W_k(s).
    \end{equation}
    Let $\mes$ denote the linear Lebesgue measure.     Compare \eqref{eq:new-mgle-repr} and \eqref{eq:mgle-repr-m}, and use
    the uniqueness of such martingale representations, in order to see that
    for all positive integers $m<n$ and $1\le k\le m$,
    \begin{equation} \label{eq:X-mgle}
   	 \ME\left( X_{n, k}(t)\mid \CF_m\right) = X_{m, k}(t)
	 \qquad\text{$(\MP\otimes\mes)$-a.e.}
    \end{equation}
        Thanks to \eqref{eq:first-conditional}, $\ME (|M_n(t)|^2) \le \ME (|M(t)|^2)$
    for every $n \ge1$ and $t \in [0\,, T].$ Since
    $\ME (|M_n(t)|^2) = \sum_{k=1}^n\ME\int_0^t[X_{n, k}(s)]^2\,\md s$
    -- see \eqref{eq:mgle-repr} --  it follows that
$$
\sup_{n \ge 1}\sum_{k=1}^n\ME\int_0^t\left[X_{n, k}(s)\right]^2\,\md s < \infty.
$$
    Consequently, there exists a $\mes$-null set $\mathcal{N}\subset[0\,,T]$ such that
    $$
    	\sup\limits_{n \ge k}\ME\left[X_{n, k}(s)\right]^2 < \infty
	\qquad\text{for every integer $k\ge1$ and all $s\in[0\,,T]\setminus\mathcal{N}$.}
    $$
    Fix a time point $t \in [0\,, T]\setminus\mathcal N$. The classical martingale convergence theorem,
    once applied to the discrete-time martingale $\{X_{n, k}(t)\}_{n \ge k}$, implies that
    \begin{equation}
    \label{eq:549}
    	\lim_{n\to\infty}X_{n, k}(t) = X_{\infty, k}(t)\qquad\mbox{a.s. and in}\ L^2;
    \end{equation}
    Therefore, for every $t\in[0\,,T]$,
    \begin{align}\label{eq:conv-X-to-X} \begin{split}
    	\sum_{k=1}^n\int_0^t\ME\left(|X_{n, k}(s) - X_{\infty, k}(s)|^2\right)\,\md s 
		&= \lim_{m \to \infty}\ME\int_0^t\left|X_{n, k}(s) - X_{m, k}(s)\right|^2\,\md s \\ 
	& = \lim_{m \to \infty}\ME\left(|M_m(t) - M_n(t)|^2\right)\\
	& = \ME\left(|M(t) - M_n(t)|^2\right)\\
	& \to 0\qquad\text{as $n \to \infty$.}
    \end{split}\end{align}
    As a result, we find that 
\begin{align*}
\lim_{n\to\infty}
    \int_0^tX_{n, k}(s)\,\md W_k(s) & = \int_0^tX_{\infty, k}(s)\,\md W_k(s)\quad \mbox{a.s. and in}\quad L^2,\\
\sum_{k=1}^{\infty}\ME\int_0^t&\left[X_{\infty, k}(s)\right]^2\,\md s < \infty.
\end{align*}
        We now show that, for all $t \in [0\,, T]$,
    \begin{equation}\label{eq:conv-L2}
    	\lim_{n\to\infty} M_n(t) = M_{\infty}(t) := \SL_{k=1}^{\infty}\int_0^tX_{\infty, k}(s)\,\md W_k(s).
    \end{equation}
    To see this, let us first write
    \begin{align} \label{eq:diff-L2} \begin{split}
    	&\ME\left(|M_n(t) - M_{\infty}(t)|^2\right) \\
	&\hskip0.5in=  \sum_{k=1}^n\ME\int_0^t\left[X_{\infty, k}(s) - X_{n, k}(s)\right]^2\,\md s  
		+ \SL_{k=n+1}^{\infty}\ME\int_0^t\left[X_{\infty, k}(s)\right]^2\,\md s.
    \end{split} \end{align}
    The first term on the right converges to $0$ as $n \to \infty$ because of~\eqref{eq:conv-X-to-X}. 
    The corresponding second term tends to $0$ as $n \to \infty$ for the following reasons:
    Because of \eqref{eq:549} and Fatou's lemma, for every $m \ge 1$ and $t\in[0\,,T]$, 
    \begin{align*}
    	\sum_{k=1}^{m}\ME\int_0^t\left[X_{\infty, k}(s)\right]^2\,\md s 
		&\le \limsup\limits_{n \to \infty}\sum_{k=1}^m
			\ME\int_0^t\left[X_{n, k}(s)\right]^2\,\md s \\
		&  = \limsup\limits_{n \to \infty}\ME (|M_n(t)|^2) \le \ME (|M(t)|^2).
    \end{align*}
    Therefore, we let $m \to \infty$ to find that
    $\sum_{k=1}^{\infty}\ME\int_0^t[X_{\infty, k}(s)]^2\,\md s \le \ME M^2(t) < \infty$,
    and hence the second term on the right-hand side of \eqref{eq:diff-L2} goes to zero as $n\to\infty$,
    as was announced. This completes the proof of~\eqref{eq:conv-L2}. Finally, we may compare \eqref{eq:M-n-M} with \eqref{eq:conv-L2} to see that
    \begin{equation}\label{eq:final-martingale}
    	M(t) = \SL_{k=1}^{\infty}\int_0^tX_{\infty, k}(s)\,\md W_k(s)
    \qquad\text{a.s.\ for every $t\in[0\,,T]$}.
    \end{equation}
    It is easy to see that both sides have continuous modifications, viewed as random processes
    indexed by $t\in[0\,,T]$. 
    Thus, we can deduce Lemma~\ref{lemma:mgle} from the preceding by applying
    \eqref{eq:final-martingale} to the continuous modifications of both sides of \eqref{eq:final-martingale}.
\end{proof}

\section*{Acknowledgements} 
This research was partially supported by  NSF grants 
DMS-1409434 (A.S.) and DMS-1608575 (D.K.). 
We thank \textsc{Sergey Bobkov, Nikolay Krylov, Soumik Pal, Timur Yastrzembsky} for helpful discussions,
and an anonymous referee for several informative comments about the history of this problem and for
the references \cite{FeyelUstunel2002,FeyelUstunel2004,Ustunel} where, among other things,  
connections are for the first time established between functional inequalities and 
Girsanov-type $L\log L$ conditions.

\end{document}